\documentclass[12pt,reqno]{amsart}

\newcommand\version{September 14, 2023}

\usepackage{amsmath,amsfonts,amsthm,amssymb,amsxtra}
\usepackage{bbm} % to get \1
\usepackage{hyperref} % creates hyperlinks to the references 
\newtheorem{theorem}{Theorem}%[section]

\newtheorem{corollary}[theorem]{Corollary}

\theoremstyle{definition}

\theoremstyle{remark}

\newtheorem{remark}[theorem]{Remark}
\newtheorem{remarks}[theorem]{Remarks}

\newcommand{\1}{\mathbbm{1}}

\newcommand{\C}{\mathbb{C}}

\renewcommand{\epsilon}{\varepsilon}

\newcommand{\loc}{{\rm loc}}

\newcommand{\N}{\mathbb{N}}

\renewcommand{\phi}{\varphi}
\newcommand{\R}{\mathbb{R}}

\newcommand{\w}{\mathrm{weak}}

\begin{document}

\title[A div-curl inequality for orthonormal functions --- \version]{A div-curl inequality for orthonormal functions}

\author{Rupert L. Frank}
\address[Rupert L. Frank]{Mathe\-matisches Institut, Ludwig-Maximilans Universit\"at M\"unchen, The\-resienstr.~39, 80333 M\"unchen, Germany, and Munich Center for Quantum Science and Technology, Schel\-ling\-str.~4, 80799 M\"unchen, Germany, and Mathematics 253-37, Caltech, Pasa\-de\-na, CA 91125, USA}
\email{r.frank@lmu.de}

\thanks{\copyright\, 2023 by the author. This paper may be reproduced, in its entirety, for non-commercial purposes.\\
	Partial support through US National Science Foundation grant DMS-1954995, as well as through the Deutsche Forschungsgemeinschaft through Germany’s Excellence Strategy EXC-2111-390814868 and through TRR 352 – Project-ID 470903074 is acknowledged.}

\dedicatory{Dedicated to Fritz Gesztesy on the occasion of his 70th birthday}

\begin{abstract}
	We prove a bound on the sum of the product of curl-free and divergence-free vector fields. Under appropriate orthonormality conditions our bound scales sublinearly in the number of terms, similar in spirit to Lieb--Thirring inequalities.
\end{abstract}

\maketitle

\section{Introduction and main result}

We are interested in pairs of vector fields $E$ and $B$ on $\R^d$ with $d\geq 2$ that satisfy $E,B\in L^2(\R^d,\R^d)$. Clearly their pointwise product $E\cdot B$ belongs to $L^1(\R^d)$. A deep observation in the area of compensated compactness is that this trivial fact can be improved under additional constraints on the curl of $E$ and the divergence of $B$, specifically under the assumptions
\begin{equation}
	\label{eq:divcurl}
	\nabla\wedge E = 0 
	\qquad\text{and}\qquad
	\nabla\cdot B = 0 \,.
\end{equation}
These equations are understood in the weak sense, that is,
$$
\int_{\R^d} ((\partial_j u) E_k - (\partial_k u) E_j)\,dx = 0
\qquad\text{for all}\ 1\leq j,k\leq d \,,
$$
and
$$
\sum_{j=1}^d \int_{\R^d} (\partial_j u) B_j\,dx = 0
$$
for all $u\in C^1_c(\R^d)$.

A theorem of Coifman, Lions, Meyer and Semmes \cite{CoLiMeSe} states that under the assumptions \eqref{eq:divcurl} the product $E\cdot B$ belongs to the Hardy space $\mathcal H^1(\R^d)$. An earlier compactness result in a related spirit, called the div-curl lemma, goes back to Murat and Tartar. 

The difference between $L^1$ and $\mathcal H^1$ is small, but of crucial importance in several applications in geometric analysis, nonlinear elasticity, homogenization, conservation laws and other fields. The definition of the Hardy space can be found, for instance, in \cite[Section 2.1]{Gr}. Since it is not relevant for what follows, we omit it. Instead we are interested here in the weaker assertion that, still under \eqref{eq:divcurl},
\begin{equation}
	\label{eq:dualsob}
	E\cdot B \in \dot W^{-1,\frac{d}{d-1}}(\R^d) \,.
\end{equation}
The space on the right side is the dual of the homogeneous Sobolev space $\dot W^{1,d}(\R^d)$, the latter being the space of equivalence classes modulo constants of elements in $L^1_\loc(\R^d)$ that are weakly differentiable with gradient in $L^d(\R^d,\R^d)$. By the Poincar\'e inequality, we have $\dot W^{1,d}(\R^d)\subset BMO(\R^d)$, the space of functions (modulo constants) of bounded mean oscillations; see \cite[Section 3.1]{Gr}. By duality \cite[Section 3.2]{Gr}, this implies $\mathcal H^1(\R^d) \subset \dot W^{-1,\frac{d}{d-1}}(\R^d)$, so \eqref{eq:dualsob} follows from the Coifman--Lions--Meyer--Semmes theorem. In the case $d=2$ the bound \eqref{eq:dualsob} appears in work of Wente \cite{We} on surfaces with prescribed mean curvature; see also \cite[Appendix]{BrCo}.

In recent years there has been some interest in extending inequalities from classical and harmonic analysis to the setting of orthonormal functions. The origin of these investigations lies in the Lieb--Thirring inequality, an extension of a certain Sobolev interpolation inequality to the setting of orthonormal functions, which was a crucial ingredient in the proof of stability of matter by Lieb and Thirring \cite{LiTh0,LiTh}. This inequality, and its generalizations, have found many applications ever since, and we refer to \cite{Fr2,FrLaWe,Fr3} for references. Among the inequalities that have been generalized to the setting of orthonormal functions are the Sobolev and Hardy--Littlewood--Sobolev inequality \cite{Li2}, the Strichartz inequality \cite{FrLeLiSe,FrSa,BeHoLeNaSa,BeLeNa}, the Stein-Tomas inequality \cite{FrSa} and Sogge's spectral clusters inequality \cite{FrSa2}.

Our goal in this paper is to present a `Lieb--Thirring' version of the above div-curl inequality. While many of the before-mentioned extensions were motivated by some specific applications, at the moment we do not know of any such application of this new result and consider it as an interesting task to find such. Instead, our motivation comes from trying to understand features of inequalities for orthonormal functions and, as will be pointed out later, the new inequality contains features that sharply distinguish it from existing ones.

To set the stage, we notice that if $(E_n),(B_n)\subset L^2(\R^d,\R^d)$ are normalized and satisfy
\begin{equation}
	\label{eq:divcurln}
	\nabla\wedge E_n = 0 
	\qquad\text{and}\qquad
	\nabla\cdot B_n = 0 
	\qquad\text{for all}\ n \,,
\end{equation}
then
\begin{equation}
	\label{eq:triangle}
	\| \sum_{n=1}^N E_n\cdot B_n \|_{\mathcal H^1} \lesssim N \,.
\end{equation}
This simply follows by the triangle inequality in $\mathcal H^1$ and the bound for a single pair of vector fields. In \eqref{eq:triangle} and below we use the notation $\lesssim$ to denote an inequality with an implicit constant that depends only on the dimension $d$. In the present paper we are not concerned with the values of these constants.

Our main result is that, if the $(E_n)$ are not only normalized, but also orthogonal, and similarly for the $(B_n)$, and if the $\mathcal H^1$-norm is replaced by the $\dot W^{-1,\frac{d}{d-1}}$-norm, then the linear growth in $N$ in \eqref{eq:triangle} can be replaced by a sublinear growth $N^{1-\frac1d}$. The precise formulation is as follows.

\begin{theorem}\label{main}
	Let $d\geq 2$. Then, if $(E_n),(B_n)$ are orthonormal systems in $L^2(\R^d,\R^d)$ satisfying \eqref{eq:divcurln}, one has
	\begin{equation}
		\label{eq:main}
		\| \sum_{n=1}^N E_n\cdot B_n \|_{\dot W^{-1,\frac{d}{d-1}}} \lesssim N^{1-\frac 1d} \,.
	\end{equation}
	If $d\geq 3$, the same is valid either without the orthogonality assumption on the $(E_n)$ or without the orthogonality assumption on the $(B_n)$.
\end{theorem}

\begin{remarks}
	(a) The following two features distinguish the above inequality from other known inequalities for orthonormal functions. First, the quantity that is bounded is bilinear rather than a sum of squares, as, for instance, in Corollary \ref{liebsob} below. Second, this quantity is not controlled in an $L^p$ norm, but rather in a norm involving (negative) smoothness.\\
	(b) The inequality \eqref{eq:main} easily generalizes to an inequality for possibly infinite orthonormal systems, namely, for sequences $\lambda =(\lambda_n)$ from the Lorentz sequence space $\ell^{\frac{d}{d-1},1}$ one has, under the previous assumptions on $(E_n),(B_n)$,
	$$
	\| \sum_n \lambda_n E_n\cdot B_n \|_{\dot W^{-1,\frac{d}{d-1}}} \lesssim \|\lambda\|_{\ell^{\frac{d}{d-1},1}} \,.
	$$
	Indeed, by multiplying, say, $E_n$ by a unimodular constant, we may assume that $\lambda_n\geq 0$. Then $\sum_n \lambda_n E_n\cdot B_n = \int_0^\infty \sum_n \1(\lambda_n>\tau) E_n\cdot B_n\,d\tau$ and, so by the triangle inequality and \eqref{eq:main},
	\begin{align*}
		\| \sum_n \lambda_n E_n\cdot B_n \|_{\dot W^{-1,\frac{d}{d-1}}} 
		& \leq \int_0^\infty \| \sum_n \1(\lambda_n>\tau) E_n\cdot B_n \|_{\dot W^{-1,\frac{d}{d-1}}}  \,d\tau \\
		& \lesssim \int_0^\infty (\#\{ n:\ \lambda_n>\tau \})^{1-\frac1d}\,d\tau = \|\lambda\|_{\ell^{\frac{d}{d-1},1}} \,.
	\end{align*}
	(c) Theorem \ref{main} raises several questions: Is the growth $N^{1-\frac1d}$ optimal and is it connected to a semiclassical limit? Is the passage from $\mathcal H^1$ to $\dot W^{-1,\frac d{d-1}}$ necessary for an $N^{1-\frac 1d}$ growth? Is the orthogonality of both $(E_n)$ and $(B_n)$ necessary in dimension $d=2$?  
\end{remarks}

The author warmly thanks Julien Sabin and Jean-Claude Cuenin for many discussions about inequalities for orthonormal systems and possible bilinear variants.

It is a pleasure to dedicate this paper to Fritz Gesztesy on the occasion of his 70th birthday, in fond memory of our editorial collaborations.

%%%%%%%%%%%%%%%%%
%%%%%%%%%%%%%%%%%

\section{Proof of Theorem \ref{main}}

In this section we give a proof of the first part of Theorem \ref{main}, namely in the case where both $(E_n)$ and $(B_n)$ satisfy the orthonormality condition. The advantage of this proof is that it works also in dimension $d=2$. The main ingredient is an endpoint Schatten class bound for commutators, which we recall first.

\subsection*{A Schatten class bound for commutators}

The Riesz transform is the vector valued operator defined by
$$
R = (-i\nabla)(-\Delta)^{-\frac 12} \,.
$$
We are interested in mapping properties of the commutator $[R,u]= Ru - uR$, where we identify the function $u$ with the operator of multiplication with this function. A famous result of Coifman, Rochberg and Weiss \cite{CoRoWe} states that $[R,u]$ is bounded (from $L^2(\R^d)$ to $L^2(\R^d,\R^d)$) if and only if $u\in BMO(\R^d)$. A corresponding result in Schatten classes $\mathcal S^p$ with $p>d$ is due to Jansson and Wolff \cite{JaWo}. A characterization in the endpoint case $p=d$ was obtained by Connes, Semmes, Sullivan and Teleman \cite[Appendix]{CoSuTe}, based on earlier work of Rochberg and Semmes \cite{RoSe}. For recent alternative and direct proof see \cite{LoMDSuZa,FrSuZa}; for a variation of the proof in \cite{CoSuTe} see \cite{Fr4}.

To state this result, we recall that given Hilbert spaces $\mathcal H$ and $\mathcal K$ and an exponent $0<p<\infty$, the weak Schatten class $\mathcal S^p_\w$ consists of all compact operators $K:\mathcal H\to\mathcal K$ whose singular values $(s_n(K))$ belong to the weak sequence space $\ell^p_\w$. The $\mathcal S^p_\w$-(quasi)norm of an operator is the $\ell^p_\w$-(quasi)norm of the squences of singular values, that is,
$$
\| K \|_{\mathcal S^p_\w} = \sup_n n^\frac1p s_n(K) \,.
$$

\begin{theorem}\label{commschatten}
	Let $d\geq 2$. Then, if $u\in\dot W^{1,d}(\R^d)$, one has $[R,u]\in \mathcal S^d_\w$ with
	$$
	\| [R,u] \|_{\mathcal S^d_\w} \lesssim \| \nabla u \|_{L^d} \,.
	$$
\end{theorem}

\begin{proof}
	As mentioned before, this result follows by combining \cite[Corollary 2.8]{RoSe} with \cite[Theorem in the appendix]{CoSuTe}. Let us show alternatively how this theorem can be deduce from the result in \cite{FrSuZa}, which concerns the operator $[\gamma\cdot R,u]$, where $\gamma_1,\ldots,\gamma_d$ are Hermitian $N\times N$ matrices, $N:=2^{[\frac d2]}$, satisfying
	$$
	\gamma_j\gamma_k + \gamma_k\gamma_j = 2\delta_{jk}
	\qquad\text{for all}\ 1\leq j,k\leq d \,.
	$$
	Moreover, we use the notation $\gamma\cdot R=\sum_{j=1}^d \gamma_j R_j$. The operator $[\gamma\cdot R,u]$ acts in $L^2(\R^d,\C^N)$ and we emphasize that the function $u$ is assumed to be scalar, that is, real- or complex-valued, and to act trivially on the factor $\C^N$. By \cite[Theorem 1.1]{FrSuZa},
	$$
	\| [\gamma\cdot R,u] \|_{\mathcal S^d_\w} \lesssim \| \nabla u \|_{L^d} \,.
	$$
	Since, by the anticommutation relations of the $\gamma$-matrices,
	$$
	[R_j,u] = \frac12 \left( [\gamma\cdot R,u]\gamma_j + \gamma_j [\gamma\cdot R,u] \right),
	$$
	this implies the corresponding bound for $[R_j,u]$, as claimed.
\end{proof}

%%%%%%%%%%%%%%%%%%%%%%%%%%%%%%%%%

\subsection*{Proof of Theorem \ref{main}}

The facts that $E_n\in L^2(\R^d,\R^d)$ and $\nabla\wedge E_n =0$ imply that there is an $f_n\in L^2(\R^d)$ such that
$$
E_n = R f_n \,,
$$
where $R$ is the Riesz transform defined before. For $u\in\dot W^{1,d}(\R^d)$, we use the assumption $\nabla\cdot B_n = 0$ to write
$$
\int_{\R^d} u E_n\cdot B_n \,dx = \int_{\R^d} \left( u R f_n \cdot B_n - R(uf_n)\cdot B_n \right)dx = - \langle [R,u] f_n, B_n \rangle \,.
$$
Here we write $\langle\cdot,\cdot\rangle$ for the inner produce in $L^2(\R^d,\R^d)$. We now use the simple fact from Hilbert space theory that, if $K$ is a compact operator from $\mathcal H$ to $\mathcal K$ and if $(x_n)$ and $(y_n)$ are orthonormal systems in $\mathcal H$ and $\mathcal K$, respectively, then
$$
\sum_{n=1}^N \left| \langle x_n, K y_n \rangle \right| \leq \sum_{n=1}^N s_n(K) \,.
$$
(This inequality appears in \cite[Theorem 11.5.6]{BiSo} in case $\mathcal H=\mathcal K$. The general case follows by composing with an isometry $\mathcal H\to\mathcal K$ when, without loss of generality, $\dim \mathcal H\leq\dim\mathcal K$.)

For any $p> 1$ and any $N\in\N$, we have
$$
\sum_{n=1}^N s_n(K) \leq \sup_n n^{\frac1p} s_n(K) \sum_{n=1}^N n^{-\frac1p} \leq \frac{p}{p-1} \| K \|_{\mathcal S^p_\w} N^{1-\frac1p} \,.
$$
The last inequality here follows by comparing the sum with an integral.

Combining the previous three relations and noting that the orthonormality of $(E_n)$ implies that of $(f_n)$, we arrive at
$$
\left| \sum_{n=1}^N \int_{\R^d} u E_n\cdot B_n\,dx \right| \leq \sum_{n=1}^N \left| \langle [R,u] f_n, B_n \rangle \right| \leq \frac{d}{d-1} \| [R,u] \|_{\mathcal S^d_\w} N^{1-\frac 1d} \,.
$$
According to Theorem \ref{commschatten}, the right side is bounded by a constant, depending only on $d$, times $\|\nabla u\|_{L^d} N^{1-\frac 1d}$. By duality, this implies the assertion of Theorem \ref{main}.\qed  

\medskip

The above proof of Theorem \ref{main} is modelled after one of the Coifman--Lions--Meyer--Semmes proofs of $E\cdot B\in\mathcal H^1$, namely where this fact is deduced from the boundedness of $[R,u]$ for $u\in BMO$.

\begin{remark}
	If instead of Theorem \ref{commschatten} one applies the Janson--Wolff result \cite{JaWo} saying that $\| [R,u] \|_{\mathcal S^p}\lesssim_p \| u \|_{\dot W^{\frac dp,p}}$ for $d<p<\infty$ (here $\mathcal S^p$ denotes the (strong) Schatten class, defined in terms of the (strong) sequence space $\ell^p$, and $\dot W^{\frac dp,p}$ denotes a homogeneous fractional Sobolev space \cite[Section 6.1]{Le}), then the same method of proof yields
	\begin{equation}
		\label{eq:maininterpol}
			\| \sum_n \lambda_n E_n\cdot B_n \|_{\dot W^{-\frac d{q'},q}} \lesssim_q \|\lambda \|_{\ell^q}
			\qquad\text{for all}\ 1<q<\frac{d}{d-1} \,, q'=\frac{q}{q-1} \,.
	\end{equation}
\end{remark}

The bound \eqref{eq:maininterpol} is somewhat intermediate between \eqref{eq:triangle} and \eqref{eq:main}, in the sense that for $\lambda_n\in\{0,1\}$ the growth in $N:=\#\{n:\ \lambda_n = 1\}$ interpolates between $N$ and $N^{1-\frac 1d}$.

%%%%%%%%%%%%%%%%%%%%%%%%%%%%%%%%%%

\section{Alternative proof of Theorem \ref{main} for $d\geq 3$}

In this section we give a proof of the second part of Theorem \ref{main}, namely that in dimension $d\geq 3$ only one of the orthogonality assumptions on $(E_n)$ and $(B_n)$ is necessary. The main ingredient is the well-known Cwikel--Lieb--Rozenblum (CLR) inequality, which we recall first.

\subsection*{The CLR inequality and its consequences}

Cwikel \cite{Cw} proved Schatten class bounds for products of multiplication operators in position and in momentum space. We only need the following particular case, which, by the Birman--Schwinger principle, is equivalent to bounds of Lieb \cite{Li} and Rozenblum \cite{Ro} on the number of negative eigenvalues of Schr\"odinger operators. We refer to \cite[Chapter 4]{FrLaWe} for details and further references.

\begin{theorem}\label{cwikel}
	Let $d\geq 3$. Then, if $u\in L^d(\R^d)$, one has $u(-\Delta)^{-\frac 12}\in\mathcal S^d_\w$ with
	$$
	\| u(-\Delta)^{-\frac12} \|_{\mathcal S^d_\w} \lesssim \| u \|_{L^d} \,.
	$$
\end{theorem}

Lieb \cite{Li2} observed that as a consequence of this theorem on obtains a Sobolev inequality for orthonormal functions. For a general statement of the underlying duality principle see \cite{FrSa}; for a direct proof of the Sobolev inequality for orthonormal functions see \cite{Ru,Fr1}. Next, we recall Lieb's proof for the sake of completeness and to underline the similarities and differences to the first proof of Theorem \ref{main}. Also, Lieb only considers functions taking values in $\C$, while we will need the case of functions taking values in $\C^M$. (Here it is slightly more natural to consider complex rather than real values.)

\begin{corollary}\label{liebsob}
	Let $d\geq 3$. Then, if $M\in\N$ and $(\psi_n)$ is an orthonormal systems in $\dot H^1(\R^d,\C^M)$, one has
	$$
	\| \sum_{n=1}^N |\psi_n|^2 \|_{L^\frac{d}{d-2}} \lesssim M^\frac2d N^{1-\frac2d} \,.
	$$
\end{corollary}

\begin{proof}
	Set $f_n:=(-\Delta)^{\frac12}\psi_n$, so that $(f_n)$ is an orthonormal system in $L^2(\R^d,\C^M)$. For $u\in L^d(\R^d,\R)$ we have, similarly as in the first proof of Theorem \ref{main},
	$$
	\sum_{n=1}^N \int_{\R^d} u^2 |\psi_n|^2 \,dx = \sum_{n=1}^N \| u (-\Delta)^{-\frac12} f_n \|_{L^2}^2 \leq \sum_{n=1}^N s_n(K)^2 \,,
	$$
	where $K:=u(-\Delta)^{-\frac12}\otimes\1_{\C^M}$, considered as an operator in $L^2(\R^d,\C^M)= L^2(\R^d)\otimes\C^M$. For $p>2$ we can bound, similarly as before,
	$$
	\sum_{n=1}^N s_n(K)^2 \leq \left( \sup_n n^\frac1p s_n(K) \right)^2 \sum_{n=1}^N n^{- \frac2p} \leq \frac{p}{p-2} \|K\|_{\mathcal S^p_\w}^2 N^{1-\frac2p} \,.
	$$
	By Theorem \ref{cwikel}, we have
	$$
	 \|K\|_{\mathcal S^d_\w} \leq M^\frac1d \| u(-\Delta)^{-\frac12} \|_{\mathcal S^d_\w} \lesssim M^\frac1d \| u \|_{L^d} \,.
	$$
	By duality, this implies the assertion of the corollary.
\end{proof}

%%%%%%%%%%%%%%%%%%%%%%%%%%%%%%%%%%%%

\subsection*{Alternative proof of Theorem \ref{main} for $d\geq 3$}

We now show that for inequality \eqref{eq:main} the orthogonality of only one of the systems $(E_n)$ and $(B_n)$ is necessary, provided $d\geq 3$. We emphasize that both systems are assumed to be normalized.

\medskip

\emph{First case.} We begin by assuming the orthogonality of the $(E_n)$. The facts that $E_n\in L^2(\R^d,\R^d)$ and $\nabla\wedge E_n =0$ imply that there is a $\phi_n\in \dot H^1(\R^d)$ such that
$$
E_n = \nabla \phi_n \,.
$$
The divergence assumption on $B_n$ implies that
$$
E_n\cdot B_n = \nabla \phi_n\cdot B_n = \nabla\cdot (\phi_n B_n) \,.
$$
Thus,
$$
\int_{\R^d} u E_n\cdot B_n\,dx = - \int_{\R^d} \nabla u\cdot (\phi_n B_n)\,dx \,.
$$
It follows from H\"older's inequality that
$$
\left| \sum_{n=1}^N \int_{\R^d} u E_n\cdot B_n\,dx \right| \leq \| \nabla u\|_{L^d} \| \sum_{n=1}^N \phi_n B_n \|_{L^\frac{d}{d-1}} \,.
$$
By the pointwise Schwarz inequality
$$
\left| \sum_{n=1}^N \phi_n B_n \right| \leq \left( \sum_{n=1}^N \phi_n^2 \right)^\frac{1}{2}\left( \sum_{n=1}^N |B_n|^2 \right)^\frac{1}{2}
$$
and the H\"older inequality (noting $\frac{2(d-1)}{d}=\frac{d-2}{d}+\frac11$), we obtain
$$
\| \sum_{n=1}^N \phi_n B_n \|_{L^\frac{d}{d-1}} \leq \| \sum_{n=1}^N \phi_n^2 \|_{L^\frac{d}{d-2}}^\frac12 \| \sum_{n=1}^N |B_n|^2 \|_{L^1}^\frac12 \,.
$$
By the normalization of $(B_n)$, we have
$$
\| \sum_{n=1}^N |B_n|^2 \|_{L^1}=N \,.
$$
Meanwhile, by Corollary \ref{liebsob}, observing that the $(\phi_n)$ are orthonormal in $\dot H^1(\R^d)$,
$$
\| \sum_{n=1}^N \phi_n^2 \|_{L^\frac{d}{d-2}} \lesssim N^{1-\frac{2}{d}} \,.
$$
We emphasize that the assumption $d\geq 3$ is needed for the latter inequality. Combining the previous bounds and using duality, we obtain the assertion of the theorem.

\medskip

\emph{Second case.}
Now assume the orthogonality of the $(B_n)$. For pedagogical reasons, we first give the proof in dimension $d=3$, where, given a vector field $F$, we identify the skew-symmetric $3\times 3$-matrix field $\nabla\wedge F$ in the usual way with a vector field.

The facts that $B_n\in L^2(\R^3,\R^3)$ and $\nabla\cdot B_n =0$ imply that there is an $A_n\in \dot H^1(\R^3,\R^3)$ such that
$$
B_n = \nabla \wedge A_n
\qquad\text{and}\qquad
\nabla\cdot A_n = 0 \,.
$$
Integrating by parts and using the curl assumption on $E_n$ we find
\begin{align*}
	\int_{\R^3} u E_n\cdot B_n\,dx & = \int_{\R^3} u E_n \cdot (\nabla\wedge A_n)\,dx = \int_{\R^3} (\nabla\wedge(u E_n))\cdot A_n \,dx \\
	& = \int_{\R^3} ((\nabla u)\wedge E_n)\cdot A_n\,dx 
	= \int_{\R^3} (\nabla u)\cdot (E_n\wedge A_n) \,dx \,.
\end{align*}
It follows from H\"older's inequality that
$$
\left| \sum_{n=1}^N \int_{\R^3} u E_n\cdot B_n\,dx \right| \leq \| \nabla u\|_{L^3} \| \sum_{n=1}^N E_n \wedge A_n \|_{L^\frac{3}{2}} \,.
$$
By the pointwise Schwarz inequality
$$
\left| \sum_{n=1}^N E_n \wedge A_n \right| \leq \left( \sum_{n=1}^N |E_n|^2 \right)^\frac{1}{2}\left( \sum_{n=1}^N |A_n|^2 \right)^\frac{1}{2}
$$
and the H\"older inequality (noting $\frac{4}{3}=\frac11 + \frac{1}{3}$), we obtain
$$
\| \sum_{n=1}^N E_n A_n \|_{L^\frac32} \leq \| \sum_{n=1}^N |E_n|^2 \|_{L^1}^\frac12 \| \sum_{n=1}^N |A_n|^2 \|_{L^3}^\frac12 \,.
$$
By the normalization of $(E_n)$, we have
\begin{equation}
	\label{eq:normalizatione}
	\| \sum_{n=1}^N |E_n|^2 \|_{L^1}=N \,.
\end{equation}
Meanwhile, by Corollary \ref{liebsob} (with $M=3$), observing that the $(A_n)$ are orthonormal in $\dot H^1(\R^d)$, states that
\begin{equation}
	\label{eq:liebsoboleva}
	\| \sum_{n=1}^N |A_n|^2 \|_{L^3} \lesssim N^{1-\frac{2}{3}} \,.
\end{equation}
The orthonormality of the $A_n$ follows by polarization from the identity
$$
\|\nabla F\|^2_{L^2} = \|\nabla\wedge F \|_{L^2}^2 + \|\nabla\cdot F\|_{L^2}^2 \,,
$$
valid for any $F\in\dot H^1(\R^d,\R^d)$.

Combining the previous bounds and using duality, we obtain the assertion of the theorem in dimension $d= 3$.

\medskip

In general dimension $d\geq 3$ the proof is most naturally carried out in the language of differential forms. For the relevant definitions and theorems we refer, for instance, to \cite{IwMa,IwScSt}. We identify a vector field $F$ on $\R^d$ with the 1-form
$$
\omega_F = F_1\,dx_1 + \cdots + F_d\,dx_d \,.
$$
The fact that $B_n$ is divergence-free translates into the equation
$$
d^* \omega_{B_n} = 0 \,,
$$
where $d^*$ is the Hodge codifferential, that is, the formal adjoint of the exterior derivative $d$. (It should be clear from the context whether the letter $d$ denotes the dimension or the exterior derivative.) By the Hodge decomposition, there is an $\alpha_n\in \dot H^1(\R^d,\Lambda^2)$ such that
$$
d^*\alpha_n = \omega_{B_n}
\qquad\text{and}\qquad
d\alpha_n = 0 \,.
$$
In terms of the Hodge star operator we can write the first equation as
$$
*\omega_{B_n} = d*\alpha_n \,.
$$
This implies the following identity of $d$-forms,
\begin{align*}
	E_n\cdot B_n\,dx = \omega_{E_n} \cdot \omega_{B_n} \,dx = \omega_{E_n} \wedge(* \omega_{B_n}) = \omega_{E_n} \wedge d*\alpha_n = - d (\omega_{E_n} \wedge *\alpha_n ) \,.
\end{align*}
In the last equality we used the fact that $E_n$ is curl-free, which translates into $d\omega_{E_n}=0$. Multiplying by $u\in\dot W^{1,d}(\R^d)$ and integrating by parts yields
$$
\int_{\R^d} u E_n\cdot B_n\,dx = - \int_{\R^d} du\wedge \omega_{E_n}\wedge *\alpha_n \,.
$$
With this identity at hand, we can argue similarly as before. By H\"older's inequality and the pointwise isometry of the Hodge star operator, we have
$$
\left| \sum_{n=1}^N \int_{\R^d} u E_n\cdot B_n\,dx \right| \leq \|\nabla u\|_{L^d} \| \sum_{n=1}^N |E_n|^2 \|_{L^2}^\frac12 \| \sum_{n=1}^N |\alpha_n|^2 \|_{L^\frac{d}{d-2}}^\frac12 \,.
$$
The normalization of $(E_n)$ gives \eqref{eq:normalizatione}. Moreover, we claim that we have the analogue of \eqref{eq:liebsoboleva} with $A_n$, $L^3$ and $\frac23$ replaced by $\alpha_n$, $L^{\frac d{d-2}}$ and $\frac2d$, respectively. Indeed, each $\alpha_n$ can be written as
$$
\alpha_n = \sum_{j<k} \alpha_n^{(j,k)} dx_j\wedge dx_k
$$
with $\alpha_n^{(j,k)}\in\dot H^1(\R^d)$ and
$$
\| B_n \|_{L^2}^2 = \|\omega_{B_n}\|_{L^2}^2 = \| d\alpha_n\|_{L^2}^2 + \| d^* \alpha_n \|_{L^2}^2 = \sum_{j<k} \| \nabla \alpha_n^{(j,k)} \|_{L^2}^2 \,.
$$
(The last equality is essentially that below \cite[Eq.\ (3.3)]{IwMa}.) Moreover,
$$
|\alpha_n|^2 = \sum_{j<k} |\alpha_n^{(j,k)}|^2 \,.
$$
Thus, we can consider $\alpha_n$ as an element of $\dot H^1(\R^d,\R^{\frac{d(d-1)}{2}})$. Polarizing the above identity for the derivatives of $\alpha_n$, we deduce from the orthonormality of the $(B_n)$ in $L^2(\R^d,\R^d)$ that of $(\alpha_n)$ in $\dot H^1(\R^d,\R^{\frac{d(d-1)}{2}})$. Thus, the claimed analogue of \eqref{eq:liebsoboleva} is a consequence of Corollary \ref{liebsob} (with $M=\frac{d(d-1)}{2}$). This completes the alternative proof of Theorem \ref{main} without the orthogonality of $(E_n)$.\qed

%%%%%%%%%%%%%%%%%%%%%%%%%%%%%%%%%%%
%%%%%%%%%%%%%%%%%%%%%%%%%%%%%%%%%%%

\bibliographystyle{amsalpha}

\end{document}